\documentclass[12pt,a4paper]{article}

\usepackage{graphicx}
\usepackage{subfig}
\usepackage{amsmath,amsthm,amssymb}
\usepackage{esint}
\usepackage{mathrsfs}

\usepackage[T1]{fontenc}

\usepackage{hyperref}

\usepackage[left=2.2cm,right=2.2cm,top=3cm,bottom=3.5cm]{geometry}

\usepackage{color}
\usepackage{array}

\theoremstyle{plain}
\newtheorem{thm}{Theorem}[section]
\newtheorem{prop}[thm]{Proposition}
\newtheorem{lemma}[thm]{Lemma}
\newtheorem{cor}[thm]{Corollary}
\theoremstyle{definition}
\newtheorem{defn}[thm]{Definition}

\theoremstyle{remark}
\newtheorem{rem}[thm]{Remark}

\numberwithin{equation}{section}

\usepackage{enumitem}
\renewcommand\labelenumi{(\alph{enumi})}
\renewcommand\theenumi\labelenumi

\newcommand{\R}{\mathbb{R}} 
\newcommand{\N}{\mathbb{N}}

\newcommand{\Grad}{\nabla}  
\newcommand{\Div}{{\rm div}\,}

\newcommand{\dx}{\,{\rm d}x}
\newcommand{\dnu}{\,{\rm d}\nu}
\newcommand{\dd}{\,{\rm d}}
\newcommand{\dlambda}{\,{\rm d}\lambda}
\newcommand{\dlambdahat}{\,{\rm d}\widehat{\lambda}}
\newcommand{\dt}{\,{\rm d}t}

\newcommand{\closure}[1]{\ov{#1}}

\newcommand{\meas}{\mathcal{M}}
\newcommand{\probmeas}{\mathcal{P}}

\renewcommand{\rho}{\varrho}
\newcommand{\sphere}[1]{\mathcal{S}^{#1}}
\newcommand{\T}{\mathbb{T}} 
\newcommand{\half}{\tfrac{1}{2}}

\newcommand{\Cc}{C_{\rm c}^\infty} 
\newcommand{\Cb}{C_{\rm b}} 

\newcommand{\weak}{{\rm weak}}
\newcommand{\ov}[1]{\overline{#1}}

\renewcommand{\subset}{\subseteq}

\renewcommand{\emptyset}{\varnothing}
\newcommand{\name}[1]{\textsc{#1}} 
\newcommand{\Var}{{\rm Var}}
\newcommand{\V}{\mathcal{V}}
\newcommand{\setGYM}{\mathbf{Y}}

\begin{document} 
	
	\title{Maximal turbulence as a selection criterion for measure-valued solutions} 
	
	\author{Christian Klingenberg\footnote{e-mail: christian.klingenberg@uni-wuerzburg.de} \and Simon Markfelder\footnote{Corresponding author; e-mail: simon.markfelder@uni-konstanz.de} \and Emil Wiedemann\footnote{e-mail: emil.wiedemann@fau.de}} 
	
	\date{\today}
	
	\maketitle
	
	\bigskip
	
	\centerline{$^{\ast}$ Julius-Maximilians-Universit\"at W\"urzburg, Institute of Mathematics,} 
	
	\centerline{Emil-Fischer-Str. 40, 97074 W\"urzburg, Germany} 
	
	\bigskip
	
	\centerline{$^{\dagger}$ Universit\"at Konstanz, Department of Mathematics and Statistics,} 
	
	\centerline{Post office box: 199, 78457 Konstanz, Germany} 
	
	\bigskip
	
	\centerline{$^{\ddagger}$ Friedrich-Alexander-Universit\"at Erlangen-N\"urnberg, Department of Mathematics,} 
	
	\centerline{Cauerstraße 11, 91058 Erlangen, Germany} 
	
	\bigskip
	
	\begin{abstract} 
		The quest for a good solution concept for the partial differential equations (PDEs) arising in mathematical fluid dynamics is an outstanding open problem. An important notion of solutions are the measure-valued solutions. It is well known that for many PDEs there exists a multitude of measure-valued solutions even if admissibility criteria like an energy inequality are imposed. Hence in recent years, people have tried to select the relevant solutions among all admissible measure-valued solutions or at least to rule out some solutions which are not relevant. 
		
		In this paper another such criterion is studied. In particular, we aim to select generalized Young measures which are ``maximally turbulent''. To this end, we look for maximizers of a certain functional, namely the variance, or more precisely, the Jensen defect of the energy. We prove existence of such a maximizer and we show that its mean value and total energy is uniquely determined. Our theory is carried out in a very general setting which may be applied in many situations where maximally turbulent measures shall be selected among a set of generalized Young measures. 
		
		Finally, we apply this general framework to the incompressible and the isentropic compressible Euler equation. Our criterion of maximal turbulence is plausible and leads to existence and uniqueness in a certain sense (in particular, the mean value and the total energy of different maximally turbulent solutions coincide). 
	\end{abstract}
	
	\bigskip
	
	
	\bigskip
	
	
	\bigskip
	
	\tableofcontents

	\section{Introduction} \label{sec:intro} 
	
	The quest for a good solution concept for the partial differential equations (PDEs) arising in mathematical fluid dynamics is an outstanding open problem. Recent developments suggest that weak solutions do not serve as a satisfactory concept even if additional requirements such as energy inequalities are imposed. On the one hand this is due to the lack of uniqueness of such \emph{admissible} weak solutions (i.e.,~weak solutions which satisfy a certain energy inequality), which has been shown by convex integration, see e.g.~\name{De~Lellis}-\name{Sz{\'e}kelyhidi}~\cite{DelSze09,DelSze10}. On the other hand, weak solutions seem to be impractical if the solution of a certain system of PDEs is understood as a limit of approximate solutions. The latter is natural as we usually consider simplified models of the real world which must be interpreted as limits of some more involved models. Examples are the Euler equations, which arise as the vanishing viscosity limit of the Navier-Stokes equations, or incompressible models, which are understood as the low Mach number limit of compressible models, etc. Another important practical example is given by sequences of approximate solutions generated by numerical schemes. 
	
	In many cases such approximative sequences exhibit oscillations and concentrations which we would like to be captured by the corresponding limit. Note that weak solutions (which are functions taking values in the phase space) are not able to capture such oscillatory and concentrative behaviour, however measure-valued solutions are. A measure-valued solution can be interpreted as a family of probability measures on the phase space parametrized by the points $x$ in the physical domain. Such families are also known as \emph{Young measures}.
	
	Measure-valued solutions as introduced by \name{DiPerna}~\cite{DiPerna85} have been studied vastly in the literature. Similar to admissibility conditions imposed for weak solutions, one usually also requires measure-valued solutions to satisfy a certain energy inequality, which leads to the notion of \emph{admissible} measure-valued solutions. Notice furthermore that in order to incorporate concentrations properly, the notion of a \emph{generalized} Young measure is necessary, see e.g.~\name{Alibert}-\name{Bouchitt{\'e}}~\cite{AliBou97} and Sect.~\ref{subsec:young-measures} below. 
	
	It is important to note that measure-valued solutions are indeed a generalization of weak solutions, in particular every weak solution is also a (very special) measure-valued solution. Moreover, admissible measure-valued solutions have other important properties which make them a plausible solution concept, e.g.~they comply with the weak-strong uniqueness principle (see e.g.~\cite{BreDelSze11,GwiWie15,WiSurvey}). 
	
	Another important fact is that it is in many cases not difficult to prove existence of admissible measure-valued solutions for any initial data, while the existence of admissible weak solutions is an open problem for many PDEs in mathematical fluid dynamics. Note furthermore that measure-valued solutions seem to be a much better concept (compared to weak solutions) when turbulent flows are studied, since in turbulence theory many claims have to be understood in an averaged or statistical sense. Genuinely measure-valued solutions which are far away from being weak solutions are not just intuitive but also numerically supported, see \name{Fjordholm}-\name{Mishra}-\name{Tadmor}~\cite{FjoMisTad16}.
	
	Nevertheless, the consideration of measure-valued instead of weak solutions does not solve the non-uniqueness problem. Quite the contrary holds, namely that there are even more measure-valued solutions than weak solutions. Consequently, mathematicians have tried to identify the relevant solutions among the possibly many admissible measure-valued solutions or -- a bit less ambitious -- to rule out solutions that are not relevant. 
	
	On the one hand, several selection criteria for admissible \emph{weak} solutions have been studied. As shown by \name{Chiodaroli}-\name{Kreml}~\cite{ChiKre14}, maximal energy dissipation, which was suggested by \name{Dafermos}~\cite{Dafermos73} as a selection criterion, does not yield the expected solution in the context of a certain class of initial data for the multidimensional isentropic Euler equations. The second author proved in \cite{Markfelder24} that a local version of \name{Dafermos}' criterion fails in a similar fashion. Also minimizing the action as suggested by \name{Gimperlein et al.}~\cite{GGKS25} does not do the trick, see \name{Markfelder}-\name{Pellhammer}~\cite{MarPel25pre}. Even vanishing viscosity does in general not serve as a proper selection principle, see e.g.~\name{Colombo}-\name{Crippa}-\name{Sorella}~\cite{ColCriSor23} or \name{Huysmans}-\name{Titi}~\cite{HuyTit25}. These results suggest that there is no unique ``right'' solution among the weak ones, but rather an inherent non-uniqueness reflected by a genuinely measure-valued solution. An extensive discussion and numerical support for such a viewpoint is given in~\cite{FjoMisTad16}.
	
	In the context of measure-valued solutions, \name{Gallenm{\"u}ller}~\cite{Gallenmueller23} proposed a criterion which discards solutions to the incompressible Euler equations as unphysical if they are not obtained as low Mach number limits of solutions to the compressible Euler system. Similarly, \name{Gallenm{\"u}ller}-\name{Wiedemann}~\cite{GalWie21} ruled out solutions to the isentropic Euler equations if they do not arise as vanishing viscosity limits from the Navier-Stokes equations. Both criteria however still allow for a multitude of solutions, i.e.,~they do not lead to uniqueness. 
	
	In \cite[Sect.~4.2]{Lasarzik22}, \name{Lasarzik} considers measure-valued solutions to the incompressible Euler equations whose mean value minimizes the energy. The reader should notice that this does not imply that the energy of such solutions is minimal. Since the energy of the mean does not have a physical meaning, it is not clear what the relevance of this criterion is. Still there exists a unique solution which satisfies this criterion. 
	
	In \name{Breit}-\name{Feireisl}-\name{Hofmanov{\'a}}~\cite{BreFeiHof20_1} a multi-step selection process is carried out to identify a unique measure-valued solution of the isentropic Euler system. More precisely, one successively minimizes a countable family of cost functionals. This yields a unique minimizer. However the selected solution (i.e.~the minimizer) strongly depends on the functionals as well as the order under which they are considered. It remains unclear which functionals and which order leads to a physically relevant solution. In particular, the dependence on the order of the functionals is counterintuitive. 
	
	The latter criterion has been improved recently by \name{Feireisl}-\name{J{\"u}ngel}-\name{Luk{\'a}{\v c}ov{\'a}}~\cite{FeiJunLuk25pre}. Here only two steps are necessary. Again there is some freedom in choosing one of the functionals that are minimized, and it remains open which functional is a good choice in order to obtain a physically relevant solution.
	
	The aim of this paper is to present another criterion which selects measure-valued solutions. Motivated by the fact that selecting weak solutions has not lead to a proper solution concept (see above), we propose the counter project, namely to select a measure-valued solution which corresponds to a multitude of weak solutions. More precisely, we look for solutions which maximize the variance. This yields solutions which are \emph{as turbulent as possible}, or in other words \emph{maximally turbulent} in the sense that they represent the most spread out collection of (non-unique) weak solutions. The concept thus endorses the non-uniqueness of weak solutions instead of aiming to identify the unique ``right'' one. Keeping the aforementioned results on the failure of selecting weak solutions in mind, and in view of the non-deterministic nature of turbulence, we assess our selection criterion to be physically relevant.
	
	Unlike the criterion studied in \cite{FeiJunLuk25pre}, our criterion of \emph{maximal turbulence} is not at all related to criteria which maximize energy dissipation like the one proposed by \name{Defermos}~\cite{Dafermos73}. In addition to that, we don't see any link between maximal turbulence and vanishing viscosity. However, our concept is compatible with imposing preceding admissibility conditions. For example, instead of looking for the maximally turbulent solution within the set of all possible measure-valued solutions, one may restrict to the subset of all solutions which arise as vanishing viscosity limits. The theory established in this paper is sufficiently general to cover even these types of approaches. 
	
	Let us illustrate our concept with the following toy example. Consider two functions $v_1,v_2$ (which may be seen as weak solutions to a certain PDE, e.g.~the incompressible Euler equations). We may understand these functions as Young measures $\delta_{v_1},\delta_{v_2}$ (i.e.,~measure-valued solutions to the PDE). We consider the convex combinations\footnote{Note that convex combinations of measure-valued solutions to the Euler equations are again measure-valued solutions, see Sect.~\ref{sec:euler} for the details.} $\nu^\tau := \tau \delta_{v_1} + (1-\tau) \delta_{v_2}$ of $\delta_{v_1},\delta_{v_2}$ (where $\tau\in [0,1]$). We are looking for the most turbulent Young measure $\nu^{\max}$ in $\{ \nu^\tau\, |\, \tau\in [0,1]\}$, i.e.~the measure which represents the most spread out collection of $\delta$'s. In other words $\nu^{\max}$ is the measure which is furthest from $\delta_{v_1},\delta_{v_2}$, i.e.
	$$
	\nu^{\max} = \nu^{1/2} = \half \delta_{v_1} + \half \delta_{v_2}. 
	$$
	This Young measure $\nu^{\max}$ is also the maximizer of the functional 
	\begin{equation} \label{eq:functional}
		\V[\nu]:= \int \Var[\nu_x] \dx,
	\end{equation}
	where $\Var[\nu_x]= \langle \nu_{x}, |\cdot |^2 \rangle - \big|\langle \nu_{x}, \cdot \rangle \big|^2$ is the variance. Indeed, a simple computation yields
	\begin{align*}
		\V[\nu^\tau] &= \int \Big[ \langle \nu^\tau_{x}, |\cdot |^2 \rangle - \big|\langle \nu^\tau_{x}, \cdot \rangle \big|^2 \Big] \dx = \int \Big[ \tau |v_1|^2 + (1-\tau) |v_2|^2 - |\tau v_1 + (1-\tau) v_2 |^2 \Big] \dx \\
		&= \tau (1-\tau) \int |v_1-v_2|^2 \dx = \tau (1-\tau) \| v_1-v_2 \|_{L^2}^2,
	\end{align*} 
	which takes (for given $v_1\neq v_2$) its maximum at $\tau=\half$. Thus, it is plausible to look for Young measures which maximize the functional \eqref{eq:functional}. 
	
	In general, the variance $\Var[\nu_x]$ in \eqref{eq:functional} needs to be replaced by the Jensen defect of the energy. In our presentation, we will be even more general and just consider the Jensen defect of a convex function $f$, which leads to a functional $\V_f$. The aim is then to find a maximizer of $\V_f$ on a given subset $M$ of the set of all generalized Young measures. Our main results are the existence (see Thm.~\ref{thm:crit1-existence}) and uniqueness of the mean value (see Thm.~\ref{thm:crit1-uniqueness-meanvalue}) of such a maximizer under some assumptions (see \ref{item:A-f-strictlyconvex}-\ref{item:A-M-compact} below). As shown in Sect.~\ref{sec:euler} below, one may take $f$ to be the energy, and $M$ to be the set of all admissible measure-valued solutions to the Euler equations. So finally we are able to show existence of ``maximally turbulent'' measure-valued solutions of the Euler equations, and that the mean value and the energy of the maximizer is uniquely determined. 
	
	In this paper we will stick to the framework established by \name{Alibert}-\name{Bouchitt{\'e}}~\cite{AliBou97} in order to describe concentrations. In particular, we will not work with the notion of a \emph{dissipative measure-valued solution} as established by \name{Feireisl} and collaborators (see e.g.~\cite{BreFeiHof20_1}). The reason for this is that the former is more general (in the sense that it applies to a large class of PDEs) while the latter is tailored to a few particular systems of PDEs, e.g.~the Euler equations. Still our theory is valid in the context of dissipative measure-valued solutions as well, see Rem.~\ref{rem:DMV-solution} below. 
	
	This paper is organized as follows. We introduce some notation and state our criterion in a general setting in Sect.~\ref{sec:general}. We also prove our main results Thms.~\ref{thm:crit1-existence} and \ref{thm:crit1-uniqueness-meanvalue}, namely that under certain assumptions (see \ref{item:A-f-strictlyconvex}-\ref{item:A-M-compact}) a maximizer exists and that its mean value is uniquely determined. In Sect.~\ref{sec:euler} we apply our criterion to the incompressible and to the isentropic compressible Euler equations. In both cases, we consider $f$ to be the energy, and $M$ to be the set of all admissible measure-valued solutions, and we show that with this choice the assumptions \ref{item:A-f-strictlyconvex}-\ref{item:A-M-compact} hold. This allows to apply the theory which we established in Sect.~\ref{sec:general}.

	\section{A maximality criterion} \label{sec:general} 
	
	\subsection{Generalized Young measures} \label{subsec:young-measures}
	
	We first introduce some basic notation, where we follow \name{Sz{\'e}kelyhidi}-\name{Wiedemann}~\cite{SzeWie12}. We denote the space of finite Radon measures on a locally compact separable metric space $X$ by $\meas(X)$. Note that $\meas(X)$ can be identified with the dual space of $C_0(X)$, where $C_0(X)$ is the completion of $C_{\rm c}(X)$ (the space of continuous functions with compact support) with respect to the supremum norm. The space of non-negative Radon measures and the space of probability measures on $X$ are denoted by $\meas^+(X)$ and $\probmeas(X)$, respectively. 
	
	For $\Omega\subset \R^n$ open or closed, $\mu\in \meas^+(\Omega)$ and $X\subset \R^m$ open or closed, a map $\nu:\Omega\to \probmeas(X)$ is weakly-$\ast$ $\mu$-measurable if 
	$$
	x\mapsto \langle \nu_x , f \rangle := \int_{X} f(z) \dnu_x(z)
	$$
	is $\mu$-measurable for any bounded Borel function $f:X\to \R$ (i.e.~the pre-image $f^{-1}(A)$ of any open subset $A\subset X$ is Borel-measurable). The space of all weakly-$\ast$ $\mu$-measurable maps from $\Omega$ into $\probmeas(X)$ is denoted by $L^\infty_\weak(\Omega,\mu;\probmeas(X))$. If $\mu$ is the Lebesgue measure, we just write $L^\infty_\weak(\Omega;\probmeas(X))$.
	
	We work with the following notion of a generalized Young measure, which goes back to \name{Alibert}-\name{Bouchitt{\'e}}~\cite{AliBou97}, see also \cite[Chap.~12]{Rindler}, \cite[Sect.~2]{BreDelSze11}, \cite[Sect.~2.2]{SzeWie12}, \cite[Sect.~3.3.1]{WiedemannPHD}. 
	
	\begin{defn}[{See \cite[Sect.~2.2]{SzeWie12}}] \label{defn:GYM}
		Let $\Omega\subset \R^n$ open and bounded ($n\in\N$) and $m\in \N$. A \emph{generalized Young measure} is a triple $(\nu,\lambda,\nu^\infty)$, where 
		\begin{itemize}
			\item $(\nu_x)_{x\in \Omega}$ is a (classical) Young measure, i.e.~a weakly-$\ast$ $\dx$-measurable family of probability measures on $\R^m$ (in short $\nu\in L^\infty_\weak(\Omega;\probmeas(\R^m))$), 
			
			\item $\lambda$ is a non-negative measure on $\closure{\Omega}$ (in short $\lambda\in \meas^+(\closure{\Omega})$), 
			
			\item $(\nu^\infty_x)_{x\in \closure{\Omega}}$ is a weakly-$\ast$ $\lambda$-measurable family of probability measures on $\sphere{m-1}$ (in short $\nu^\infty\in L^\infty_\weak(\closure{\Omega},\lambda;\probmeas(\sphere{m-1}))$),
		\end{itemize}
		which satisfies\footnote{In the context of measure-valued solutions, some authors do not explicitly require \eqref{eq:GYM-bound}. However, in those cases \eqref{eq:GYM-bound} follows from the energy inequality, see also Rem.~\ref{rem:EnIneq->GYMbound}.} 
		\begin{equation} \label{eq:GYM-bound}
			\int_\Omega \langle \nu_x, |\cdot |^2\rangle \dx + \lambda(\closure{\Omega})<\infty.
		\end{equation} 
		We call $\nu,\lambda,\nu^\infty$ \emph{oscillation measure}, \emph{concentration measure} and \emph{concentration-angle measure}, respectively. We denote the set of all generalized Young measures by $\setGYM$, i.e.
		$$
		\setGYM := \left\{ (\nu,\lambda,\nu^\infty) \in L^\infty_\weak(\Omega; \probmeas(\R^m)) \times \meas^+(\closure{\Omega}) \times L^\infty_\weak(\closure{\Omega},\lambda; \probmeas(\sphere{m-1}))\, \Big|\, \text{\eqref{eq:GYM-bound} holds} \right\}.
		$$  
		We endow $\setGYM$ with the usual weak-$\ast$ topology, i.e.~a sequence $(\nu^k,\lambda^k,(\nu^\infty)^k)_{k\in \N}\subset \setGYM$ converges to $(\nu,\lambda,\nu^\infty)\in \setGYM$ if and only if
		\begin{equation} \label{eq:convergence}
			\langle \nu^k_x, f \rangle \dx + \langle (\nu^\infty)^k_x, f^\infty \rangle \lambda^k \mathop{\rightharpoonup}\limits^{\ast} \langle \nu_x, f \rangle \dx + \langle \nu^\infty_x, f^\infty \rangle \lambda \quad \text{ in }\meas(\closure{\Omega}) \text{ for all }f\in \mathcal{F}_2(\R^m). 
		\end{equation}
		Here 
		$$
		\mathcal{F}_2(\R^m) := \left\{ f\in C(\R^m) \,\Big| \, \exists f_0\in \mathcal{A}(\R^m) \text{ s.t.~} f=(1+|\cdot|^2) f_0 \right\},
		$$
		where 
		$$
		\mathcal{A}(\R^m) := \left\{ f_0\in \Cb(\R^m) \,\Big| \, \lim_{s\to \infty} f_0(sz) \text{ exists and is continuous in } z\in \sphere{m-1} \right\}.
		$$
		The \emph{$2$-recession function $f^\infty$} of $f\in \mathcal{F}_2(\R^m)$ is defined by 
		$$
		f^\infty: \sphere{m-1} \to \R, \quad f^\infty (z) := \lim_{s\to \infty} f_0(sz) = \lim_{s\to \infty} \frac{f(sz)}{1+s^2}. 
		$$
	\end{defn} 
	
	Note that the $2$-recession function $f^\infty$, which represents the behavior of $f$ at infinity, exists and is continuous on $\sphere{m-1}$ according to the definition of $\mathcal{A}(\R^m)$.
	
	Note furthermore that \eqref{eq:convergence} means 
	\begin{align*}
		&\int_\Omega \phi(x) \langle \nu^k_x, f \rangle \dx + \int_{\closure{\Omega}} \phi(x) \langle (\nu^\infty)^k_x, f^\infty \rangle \dlambda^k(x) \\
		&\to \int_\Omega \phi(x) \langle \nu_x, f \rangle \dx + \int_{\closure{\Omega}} \phi(x) \langle \nu^\infty_x, f^\infty \rangle \dlambda(x) & &\text{ for all }\phi\in C(\closure{\Omega}),\ f\in \mathcal{F}_2(\R^m).
	\end{align*} 
	
	\begin{rem} \label{rem:p}
		The reader should notice that one could be more general by working with an arbitrary growth factor $p\in [1,\infty)$ instead of restricting to the case $p=2$. Then one needs to replace the integral in \eqref{eq:GYM-bound} by $\int_\Omega \langle \nu_x, |\cdot |^p\rangle \dx$, and one has to consider the set
		$$ 
		\mathcal{F}_p(\R^m) = \left\{ f\in C(\R^m) \,\Big| \, \exists f_0\in \mathcal{A}(\R^m) \text{ s.t.~} f=(1+|\cdot|^p) f_0 \right\}.
		$$
		In \cite{Rindler} $p$ is chosen to be $1$; in the context of the (incompressible) Euler equations, see \cite{BreDelSze11,SzeWie12,WiedemannPHD}, one usually considers $p=2$. For our purposes the case $p=2$ suffices as well. Still we would like to emphasize that our theory holds for other choices of $p\in [1,\infty)$, too.
	\end{rem}
	
	Next we state what a convex combination of two generalized Young measures is. 
	
	\begin{defn} \label{defn:convexity}
		Let $\tau \in [0,1]$. The convex combination $(\widehat{\nu},\widehat{\lambda},\widehat{\nu}^\infty)\in \setGYM$ of two generalized Young measures $\big(\nu^1,\lambda^1,(\nu^\infty)^1\big), \big(\nu^2,\lambda^2,(\nu^\infty)^2\big) \in\setGYM$ is given by
		\begin{equation} \label{eq:defn-convexity} 
			\langle \widehat{\nu}_x , f \rangle \dx + \langle \widehat{\nu}^\infty_x , f^\infty \rangle \widehat{\lambda} = \tau \Big( \langle \nu^1_x , f \rangle \dx + \langle (\nu^\infty)^1_x , f^\infty \rangle \lambda^1 \Big) + (1-\tau) \Big( \langle \nu^2_x , f \rangle \dx + \langle (\nu^\infty)^2_x , f^\infty \rangle \lambda^2 \Big), 
		\end{equation}
		for all $f\in \mathcal{F}_2(\R^m)$ and a.e.~$x\in \Omega$. 
	\end{defn}
	
	The following lemma clarifies Defn.~\ref{defn:convexity}. 
	
	\begin{lemma} \label{lemma:convexity}
		For all $\tau \in [0,1]$ and all $\big(\nu^1,\lambda^1,(\nu^\infty)^1\big), \big(\nu^2,\lambda^2,(\nu^\infty)^2\big) \in\setGYM$ it holds that 
		\begin{align}
			\widehat{\nu} = \tau \nu^1 + (1-\tau)\nu^2, \label{eq:lemma-convexity-osc} \\
			\widehat{\lambda} = \tau\lambda^1 + (1-\tau)\lambda^2 . \label{eq:lemma-convexity-conc}
		\end{align} 
		In other words the oscillation measure $\widehat{\nu}$ and the concentration measure $\widehat{\lambda}$ of a convex combination of two generalized Young measures $\big(\nu^1,\lambda^1,(\nu^\infty)^1\big), \big(\nu^2,\lambda^2,(\nu^\infty)^2\big) \in\setGYM$ are indeed the convex combinations of the oscillation measures $\nu^1,\nu^2$ and the concentration measures $\lambda^1,\lambda^2$ respectively.  
	\end{lemma}
	
	\begin{proof} 
		In order to show \eqref{eq:lemma-convexity-osc}, it suffices to prove that 
		$$
		\langle \widehat{\nu}_x , f \rangle = \tau \langle \nu^1_x , f \rangle + (1-\tau) \langle \nu^2_x , f \rangle \qquad \text{ for all }f\in\Cb(\R^m)\text{ and a.e. }x\in \Omega.
		$$ 
		The latter follows immediately from \eqref{eq:defn-convexity} since $f^\infty\equiv 0$ for any $f\in\Cb(\R^m)$. The choice $f= |\cdot|^2$ (which means $f^\infty\equiv 1$) together with \eqref{eq:lemma-convexity-osc} yields \eqref{eq:lemma-convexity-conc}.
	\end{proof}

	\subsection{The functional $\V_f$ and its properties} 
	
	Next we consider a convex, non-negative function $f\in \mathcal{F}_2(\R^m)$. Then we define the functional $\V_f$ as follows. 
	
	\begin{defn} \label{defn:V} 
		For a generalized Young measure $(\nu,\lambda,\nu^\infty)\in\setGYM$, we set 
		$$
		\V_f[\nu,\lambda,\nu^\infty] = \int_\Omega \Big[\langle \nu_{x}, f \rangle - f\big(\langle \nu_{x}, \cdot \rangle \big) \Big] \dx + \int_{\closure{\Omega}} \langle \nu_x^\infty , f^\infty \rangle \dlambda(x).
		$$ 
	\end{defn} 
	
	\begin{rem} \label{rem:example-variance} 
		As an example, one could consider $f=|\cdot|^2$. Then $f^\infty\equiv 1$ and hence
		$$
		\V_f[\nu,\lambda,\nu^\infty] = \int_\Omega \Var[\nu_x] \dx + \lambda(\closure{\Omega})
		$$
		with the variance $\Var[\nu_x]= \langle \nu_{x}, |\cdot |^2 \rangle - \big|\langle \nu_{x}, \cdot \rangle \big|^2$. When applying our theory to the Euler equations (see Sect.~\ref{sec:euler} below), the most natural choice for $f$ is the energy. Note that the energy in the incompressible setting is given by $\half |\cdot|^2$, see Sect.~\ref{subsec:incomp-euler} below. So (up to a factor $\half$) the choice $f=|\cdot|^2$, which leeds to the variance, see above, coincides with choosing the energy.
	\end{rem}
	
	We observe that $\V_f$ takes values in $[0,\infty)$, which is the content of the following proposition.
	
	\begin{prop} \label{prop:bounds-V}
		It holds that 
		$$
		0\leq \V_f[\nu,\lambda,\nu^\infty]<\infty \qquad \text{ for all } (\nu,\lambda,\nu^\infty)\in \setGYM.
		$$
	\end{prop} 
	
	\begin{proof} 
		Since $f\in \mathcal{F}_2(\R^m)$, there exists $C>0$ such that $f(z)\leq C (1+|z|^2)$ for all $z\in \R^m$. Consequently, \eqref{eq:GYM-bound} implies 
		$$
		\int_\Omega \langle \nu_{x}, f \rangle \dx \leq C \int_\Omega \Big[ \langle \nu_{x}, 1 \rangle + \langle \nu_{x}, |\cdot |^2 \rangle \Big] \dx < \infty.
		$$
		As $f$ is non-negative, we simply find 
		$$
		-\int_\Omega f\big(\langle \nu_{x}, \cdot \rangle \big) \dx \leq 0.
		$$
		Finally 
		$$
		\int_{\closure{\Omega}} \langle \nu_x^\infty , f^\infty \rangle \dlambda(x) \leq \Big( \max_{\theta\in \sphere{m-1}}f^\infty(\theta) \Big) \lambda(\closure{\Omega}) < \infty,
		$$
		so we have shown that $\V_f[\nu,\lambda,\nu^\infty]$ is finite for any $(\nu,\lambda,\nu^\infty)\in \setGYM$.
		
		For the lower bound we invoke Jensen's inequality, which yields 
		$$
		f\big( \langle \nu_{x}, \cdot \rangle \big) \leq \langle \nu_{x}, f \rangle \qquad \text{ for a.e. } x\in \Omega,
		$$ 
		and thus $\V_f[\nu,\lambda,\nu^\infty]\geq 0$ for all $(\nu,\lambda,\nu^\infty)\in\setGYM$ as desired. 
	\end{proof}
	
	Next, we prove some important properties of the map $\V_f: \setGYM \to [0,\infty)$. 
	
	\begin{lemma} \label{lemma:properties-V} 
		The functional $\V_f: \setGYM \to [0,\infty)$ has the following properties:
		\begin{enumerate} 
			\item \label{item:prop-V-concave} $\V_f$ is concave, i.e.~for all $(\nu^1,\lambda^1,(\nu^\infty)^1),(\nu^2,\lambda^2,(\nu^\infty)^2) \in \setGYM$ and $\tau\in [0,1]$ it holds that  
			\begin{equation} \label{eq:V-concave} 
				\V_f [\widehat{\nu},\widehat{\lambda},\widehat{\nu}^\infty] \geq \tau \V_f [\nu^1,\lambda^1,(\nu^\infty)^1] + (1-\tau) \V_f [\nu^2,\lambda^2,(\nu^\infty)^2],
			\end{equation}
			with the convex combination $(\widehat{\nu},\widehat{\lambda},\widehat{\nu}^\infty)$ of $(\nu^1,\lambda^1,(\nu^\infty)^1),(\nu^2,\lambda^2,(\nu^\infty)^2)$ as defined in Defn.~\ref{defn:convexity}. \\
			As soon as $f$ is \emph{strictly} convex and $\tau\in (0,1)$, equality in \eqref{eq:V-concave} holds if and only if $\langle \nu^1_{x} , \cdot\rangle = \langle \nu^2_{x} , \cdot\rangle$ for a.e. $x\in \Omega$. 
			
			\item \label{item:prop-V-usc} $\V_f$ is upper semi-continuous with respect to the weak-$\ast$ topology, i.e.~for all sequences $(\nu^k,\lambda^k,(\nu^\infty)^k)_{k\in \N}\subset \setGYM$ which converge to some $(\nu,\lambda,\nu^\infty)\in \setGYM$ in the weak-$\ast$ topology it holds that
			\begin{equation} \label{eq:V-semi-continuous}
				\limsup_{k\to\infty} \V_f[\nu^k,\lambda^k,(\nu^\infty)^k] \leq \V_f[\nu,\lambda,\nu^\infty].
			\end{equation}
		\end{enumerate}	
	\end{lemma}
	
	\begin{proof} 
		Let $(\nu^1,\lambda^1,(\nu^\infty)^1),(\nu^2,\lambda^2,(\nu^\infty)^2) \in \setGYM$ two generalized Young measures, and $\tau\in [0,1]$. Using Lemma~\ref{lemma:convexity}, we find 
		\begin{align*}
			\V_f [\widehat{\nu},\widehat{\lambda},\widehat{\nu}^\infty] &= \int_\Omega \Big[\langle \tau \nu^1_{x} + (1-\tau) \nu^2_x, f \rangle - f\big(\langle \tau \nu^1_{x} + (1-\tau) \nu^2_{x}, \cdot \rangle \big) \Big] \dx \\
			&\qquad + \tau \int_{\closure{\Omega}} \langle (\nu^\infty)_x^1 , f^\infty \rangle \dlambda^1(x) + (1-\tau) \int_{\closure{\Omega}} \langle (\nu^\infty)_x^2 , f^\infty \rangle \dlambda^2(x)\\
			&= \int_\Omega \Big[ \tau \langle \nu^1_{x}, f \rangle + (1-\tau) \langle \nu^2_x, f \rangle - f\big(\tau \langle \nu^1_{x}, \cdot \rangle + (1-\tau) \langle \nu^2_{x}, \cdot \rangle \big) \Big] \dx \\
			&\qquad + \tau \int_{\closure{\Omega}} \langle (\nu^\infty)_x^1 , f^\infty \rangle \dlambda^1(x) + (1-\tau) \int_{\closure{\Omega}} \langle (\nu^\infty)_x^2 , f^\infty \rangle \dlambda^2(x).
		\end{align*}
		The convexity of $f$ yields 
		\begin{equation} \label{eq:201}
			f\big(\tau \langle \nu^1_{x}, \cdot \rangle + (1-\tau) \langle \nu^2_{x}, \cdot \rangle \big) \leq \tau f\big(\langle \nu^1_{x}, \cdot \rangle\big) + (1-\tau) f\big(\langle \nu^2_{x}, \cdot \rangle \big)\qquad \text{ for a.e. }x\in\Omega,
		\end{equation}
		and hence 
		\begin{align*}
			\V_f [\widehat{\nu},\widehat{\lambda},\widehat{\nu}^\infty] &\geq \int_\Omega \Big[ \tau \langle \nu^1_{x}, f \rangle + (1-\tau) \langle \nu^2_x, f \rangle - \tau f\big(\langle \nu^1_{x}, \cdot \rangle\big) - (1-\tau) f\big(\langle \nu^2_{x}, \cdot \rangle \big) \Big] \dx \\
			&\qquad + \tau \int_{\closure{\Omega}} \langle (\nu^\infty)_x^1 , f^\infty \rangle \dlambda^1(x) + (1-\tau) \int_{\closure{\Omega}} \langle (\nu^\infty)_x^2 , f^\infty \rangle \dlambda^2(x) \\
			&= \tau \V_f [\nu^1,\lambda^1,(\nu^\infty)^1] + (1-\tau) \V_f [\nu^2,\lambda^2,(\nu^\infty)^2].
		\end{align*}
		
		Now let $f$ be \emph{strictly} convex and $\tau\in (0,1)$. In this case, we have equality in \eqref{eq:201} if and only if $\langle \nu^1_{x} , \cdot\rangle = \langle \nu^2_{x} , \cdot\rangle$. Consequently, equality in \eqref{eq:V-concave} holds if and only if $\langle \nu^1_{x} , \cdot\rangle = \langle \nu^2_{x} , \cdot\rangle$ for a.e.~$x\in \Omega$ as desired. 
		
		It remains to show \ref{item:prop-V-usc}. To this end, consider a sequence $(\nu^k,\lambda^k,(\nu^\infty)^k)_{k\in \N}\subset \setGYM$ which converges to $(\nu,\lambda,\nu^\infty)\in \setGYM$ in the weak-$\ast$ topology. We first observe that this immediately yields 
		$$
		\int_\Omega \langle \nu_{x}^k, f \rangle \dx + \int_{\closure{\Omega}} \langle (\nu^\infty)_x^k , f^\infty \rangle \dlambda^k(x) \to \int_\Omega \langle \nu_{x}, f \rangle \dx + \int_{\closure{\Omega}} \langle \nu_x^\infty , f^\infty \rangle \dlambda(x).
		$$ 
		Moreover, we observe that $(\nu^k,\lambda^k,(\nu^\infty)^k) \mathop{\rightharpoonup}\limits^\ast (\nu,\lambda,\nu^\infty)$ implies
		$$ 
		\langle \nu^k_x, \cdot\rangle \mathop{\rightharpoonup}\limits^{\ast} \langle \nu_x, \cdot\rangle \qquad \text{ in }L^\infty(\Omega). 
		$$
		Together with the convexity of $f$, this leads to
		$$
		\liminf_{k\to\infty} \int_\Omega f\big(\langle \nu^k_{x}, \cdot \rangle \big) \dx \geq \int_\Omega f\big(\langle \nu_{x}, \cdot \rangle \big) \dx,
		$$
		see e.g.~\cite[Lemma~A.7.3]{Markfelder} for more details. Collecting everything, we find \eqref{eq:V-semi-continuous} as desired.
	\end{proof}
	
	\begin{rem} 
		Note that $\V_f$ is not strictly concave. Indeed take $(\nu^1,\lambda^1,(\nu^\infty)^1),(\nu^2,\lambda^2,(\nu^\infty)^2) \in \setGYM$ with $\nu^1 \neq \nu^2$ but $\langle \nu^1_{x} , \cdot\rangle = \langle \nu^2_{x} , \cdot\rangle$ for a.e.~$x\in \Omega$. Then according to Lemma~\ref{lemma:properties-V}~\ref{item:prop-V-concave}, inequality \eqref{eq:V-concave} holds with equality for $\tau\in (0,1)$.
	\end{rem}

	\subsection{The criterion and its properties} 
	
	Our goal is to define a criterion which selects particular generalized Young measures from a subset $M\subset \setGYM$. We will assume the following four assumptions regarding the function $f$ and the subset $M$.
	
	\begin{enumerate}[label=(A\arabic*)]
		\item \label{item:A-f-strictlyconvex} The function $f$ is \emph{strictly} convex.
		
		\item \label{item:A-M-nonempty} The set $M$ is non-empty, i.e.~$M\neq\emptyset$.
		
		\item \label{item:A-M-convex} The set $M$ is convex, i.e.~for all $\tau \in [0,1]$ and all $\big(\nu^1,\lambda^1,(\nu^\infty)^1\big), \big(\nu^2,\lambda^2,(\nu^\infty)^2\big) \in M$
		the convex combination $(\widehat{\nu},\widehat{\lambda},\widehat{\nu}^\infty)$ as defined in Defn.~\ref{defn:convexity} lies in $M$.
		
		\item \label{item:A-M-compact} The set $M$ is sequentially compact with respect to the weak-$\ast$ topology, i.e.~any sequence $(\nu^k,\lambda^k,(\nu^\infty)^k)_{k\in \N}\subset M$ has a converging subsequence with limit in $M$.
	\end{enumerate} 
	
	Now we are ready to formulate our maximality criterion.
	
	\begin{defn} \label{defn:crit1} 
		Let $M\subset \setGYM$. We call a generalized Young measure $(\nu,\lambda,\nu^\infty)\in M$ \emph{maximal} if $(\nu,\lambda,\nu^\infty)$ is a maximizer of $\V_f$, i.e. if 
		$$
		\V_f[\nu,\lambda,\nu^\infty] = \max_{(\widetilde{\nu},\widetilde{\lambda},\widetilde{\nu}^\infty)\in M} \V_f[\widetilde{\nu},\widetilde{\lambda},\widetilde{\nu}^\infty].
		$$
	\end{defn}
	
	\subsubsection{Existence}
	
	In this subsection we prove existence of a maximal generalized Young measure.
	
	\begin{thm} \label{thm:crit1-existence}
		Under the assumptions~\ref{item:A-M-nonempty} and \ref{item:A-M-compact}, there exists $(\nu,\lambda,\nu^\infty)\in M$ which is maximal in the sense of Defn.~\ref{defn:crit1}.
	\end{thm}
	
	The statement of Thm.~\ref{thm:crit1-existence} follows from a standard argument used in optimization theory. Still we present a detailed proof for the sake of completeness. 
	
	\begin{proof} 
		As a first step, we show that $\V_f$ is bounded on $M$. Assume $\V_f$ was not bounded on $M$. Then there exists a sequence $(\nu^k,\lambda^k,(\nu^\infty)^k)_{k\in \N}\subset M$ with 
		$$
		\V_f[\nu^k,\lambda^k,(\nu^\infty)^k] \geq k \qquad \text{ for all }k\in \N.
		$$
		By compactness (see assumption~\ref{item:A-M-compact}), we may assume that $(\nu^k,\lambda^k,(\nu^\infty)^k)_{k\in \N}$ converges to some $(\nu,\lambda,\nu^\infty)\in M$ with respect to the weak-$\ast$ topology. Then upper semicontinuity (see Lemma~\ref{lemma:properties-V}~\ref{item:prop-V-usc}) yields 
		$$
		\V_f[\nu,\lambda,\nu^\infty] \geq \limsup_{k\to\infty} \V_f[\nu^k,\lambda^k,(\nu^\infty)^k] = \infty,
		$$
		which contradicts Prop.~\ref{prop:bounds-V}. Hence $\V_f$ must be bounded on $M$.
		
		Together with the fact that $M\neq \emptyset$ (see assumption~\ref{item:A-M-nonempty}), we infer that 
		$$
		S:= \sup_{(\widetilde{\nu},\widetilde{\lambda},\widetilde{\nu}^\infty)\in M} \V_f[\widetilde{\nu},\widetilde{\lambda},\widetilde{\nu}^\infty]
		$$ 
		is finite. Consequently there exists a sequence $(\nu^k,\lambda^k,(\nu^\infty)^k)_{k\in \N}\subset M$ with 
		$$
		S-\frac{1}{k} \leq \V_f[\nu^k,\lambda^k,(\nu^\infty)^k] \leq S \qquad \text{ for all }k\in \N.
		$$
		Arguing as above, we may assume that $(\nu^k,\lambda^k,(\nu^\infty)^k)_{k\in \N}$ converges to some $(\nu,\lambda,\nu^\infty)\in M$ with respect to the weak-$\ast$ topology. Again upper semicontinuity leads to 
		$$
		\V_f[\nu,\lambda,\nu^\infty] \geq \limsup_{k\to\infty} \V_f[\nu^k,\lambda^k,(\nu^\infty)^k] \geq \limsup_{k\to\infty} \left( S-\frac{1}{k} \right) = S. 
		$$
		Thus, $(\nu,\lambda,\nu^\infty)$ is a desired maximizer. 
	\end{proof}

	\subsubsection{Mean value of maximal measure is unique} 
	
	Maximal measures are not necessarily unique, but the barycenter of their oscillation measure is. This is the content of the following proposition. 
	
	\begin{thm} \label{thm:crit1-uniqueness-meanvalue} 
		Suppose the assumptions~\ref{item:A-f-strictlyconvex} and \ref{item:A-M-convex} hold. Let furthermore 
		$$
		(\nu^1,\lambda^1,(\nu^\infty)^1) , (\nu^2,\lambda^2,(\nu^\infty)^2) \in M
		$$ 
		two generalized Young measures which are both maximal in the sense of Defn.~\ref{defn:crit1}. 
		Then 
		\begin{align}
			\langle \nu^1_{x}, \cdot \rangle &= \langle \nu^2_{x}, \cdot \rangle \qquad \text{ for a.e. } x\in \Omega, \text{ and } \label{eq:203} \\
			\int_\Omega \langle \nu^1_{x}, f \rangle \dx + \int_{\closure{\Omega}} \langle (\nu^\infty)_x^1 , f^\infty \rangle \dlambda^1(x) &= \int_\Omega \langle \nu^2_{x}, f \rangle \dx + \int_{\closure{\Omega}} \langle (\nu^\infty)_x^2 , f^\infty \rangle \dlambda^2(x). \label{eq:204}
		\end{align}
	\end{thm}
	
	\begin{proof} 
		For $\tau\in (0,1)$, consider the convex combination $(\widehat{\nu},\widehat{\lambda},\widehat{\nu}^\infty)\in \setGYM$ of $\big(\nu^1,\lambda^1,(\nu^\infty)^1\big)$, $\big(\nu^2,\lambda^2,(\nu^\infty)^2\big)$ as defined in Defn.~\ref{defn:convexity}. According to assumption~\ref{item:A-M-convex} it holds that $(\widehat{\nu},\widehat{\lambda},\widehat{\nu}^\infty)\in M$. Moreover, we deduce from assumption~\ref{item:A-f-strictlyconvex} and Lemma~\ref{lemma:properties-V}~\ref{item:prop-V-concave} that 
		$$
		\V_f[\widehat{\nu},\widehat{\lambda},\widehat{\nu}^\infty] \geq \tau \V_f[\nu^1,\lambda^1,(\nu^\infty)^1] + (1-\tau) \V_f[\nu^2,\lambda^2,(\nu^\infty)^2]
		$$
		with equality if and only if $\langle \nu^1_{x} , \cdot\rangle = \langle \nu^2_{x} , \cdot\rangle$ for a.e.~$x\in \Omega$. Since 
		$$
		\V_f[\nu^1,\lambda^1,(\nu^\infty)^1]= \V_f[\nu^2,\lambda^2,(\nu^\infty)^2] = \max_{(\widetilde{\nu},\widetilde{\lambda},\widetilde{\nu}^\infty)\in M} \V_f[\widetilde{\nu},\widetilde{\lambda},\widetilde{\nu}^\infty],
		$$
		we infer 
		$$
		\V_f[\widehat{\nu},\widehat{\lambda},\widehat{\nu}^\infty]= \V_f[\nu^1,\lambda^1,(\nu^\infty)^1]= \V_f[\nu^2,\lambda^2,(\nu^\infty)^2],
		$$
		and consequently $\langle \nu^1_{x} , \cdot\rangle = \langle \nu^2_{x} , \cdot\rangle$ for a.e.~$x\in \Omega$. 
		
		The latter yields, together with $\V_f[\nu^1,\lambda^1,(\nu^\infty)^1]= \V_f[\nu^2,\lambda^2,(\nu^\infty)^2]$, that even \eqref{eq:204} holds.
	\end{proof}
	
	\begin{rem}
		Thm.~\ref{thm:crit1-uniqueness-meanvalue} not only states that the barycenter of the oscillation measure of maximal generalized Young measures is unique (see \eqref{eq:203}), but also that the expression 
		\begin{equation} \label{eq:210}
			\int_\Omega \langle \nu_{x}, f \rangle \dx + \int_{\closure{\Omega}} \langle \nu^\infty_x , f^\infty \rangle \dlambda(x)
		\end{equation}
		is unique (see \eqref{eq:204}). In the context of the Euler equations, see Sect.~\ref{sec:euler} below, we will choose $f$ to be the energy, and hence \eqref{eq:210} will represent the total energy (including the ``concentration energy'').
	\end{rem}

	\section{Application to measure-valued solutions of the Euler equations} \label{sec:euler} 
	
	\subsection{Incompressible Euler equations} \label{subsec:incomp-euler} 
	
	The incompressible Euler equations read
	\begin{align}
		\Div v &= 0 , \label{eq:euler-mass} \\
		\partial_t v + \Div (v \otimes v) + \Grad p &= 0, \label{eq:euler-mom}
	\end{align}
	with unknown velocity $v:[0,T) \times \T^d \to \R^d$, where $d\in \N$ is the space dimension, $\T^d$ denotes the $d$-dimensional torus and $T\in(0,\infty)$. The energy density for the incompressible Euler system \eqref{eq:euler-mass}, \eqref{eq:euler-mom} is simply given by the kinetic energy $\half|v|^2$. Note furthermore that the pressure $p$ can be recovered from $v$ via the Poisson equation 
	$$
	\Delta p = - \Div \Div (v\otimes v).
	$$
	
	The notion of measure-valued solutions to the incompressible Euler equations \eqref{eq:euler-mass}, \eqref{eq:euler-mom} goes back to \name{DiPerna}-\name{Majda}~\cite{DipMaj87} (which is in turn built upon \name{DiPerna}~\cite{DiPerna85}). In the \name{Alibert}-\name{Bouchitt{\'e}}~\cite{AliBou97} framework, measure-valued solutions to \eqref{eq:euler-mass}, \eqref{eq:euler-mom} are defined as follows, see \cite[Defn.~1]{BreDelSze11}, \cite[Defn.~8]{SzeWie12} or \cite[Defn.~3.7]{WiedemannPHD}.
	
	\begin{defn}[See e.g.~{\cite[Defn.~8~a]{SzeWie12}}] \label{defn:euler-mvs}
		A generalized Young measure $(\nu,\lambda,\nu^\infty)\in \setGYM$ on $\R^d$ with parameters in $(0,T)\times \T^d$ (i.e.~$\Omega=(0,T)\times \T^d$, $n=d+1$, $m=d$ in Defn.~\ref{defn:GYM}) is called \emph{measure-valued solution} of the incompressible Euler equations \eqref{eq:euler-mass}, \eqref{eq:euler-mom} with initial data $v_0\in L^2(\T^d)$ if\footnote{In the context of the incompressible Euler equations \eqref{eq:euler-mass}, \eqref{eq:euler-mom}, we use $v\in \R^d$ and $\theta\in \sphere{d-1}$ as dummy variables when integrating with respect to $\nu_{t,x}$ and $\nu^\infty_{t,x}$, respectively.}
		\begin{align} 
			&\int_0^T \int_{\T^d} \langle \nu_{t,x} , v\rangle \cdot \Grad \psi \dx\dt = 0, \label{eq:euler-mvs-mass} \\
			&\int_0^T \int_{\T^d} \Big[ \langle \nu_{t,x}, v\rangle \cdot \partial_t \phi + \langle \nu_{t,x} , v\otimes v \rangle : \Grad \phi \Big] \dx\dt \notag\\
			&\quad + \int_0^T \int_{\T^d} \langle \nu^\infty_{t,x} , \theta\otimes \theta \rangle : \Grad \phi \dlambda(t,x) + \int_{\T^d} v_0 \cdot \phi(0,\cdot)\dx = 0 \label{eq:euler-mvs-mom}
		\end{align}
		for all test functions $\psi\in \Cc([0,T)\times \T^d)$, $\phi\in \Cc([0,T)\times \T^d;\R^d)$ with $\Div \phi = 0$. 
	\end{defn}
	
	\begin{defn}[See e.g.~{\cite[Defn.~8~b]{SzeWie12}}] \label{defn:euler-adm-mvs}
		A measure-valued solution of the incompressible Euler equations \eqref{eq:euler-mass}, \eqref{eq:euler-mom} $(\nu,\lambda,\nu^\infty)\in \setGYM$ with initial data $v_0\in L^2(\T^d)$ is called \emph{admissible} if the following assertions hold:
		\begin{itemize}
			\item The concentration measure $\lambda$ admits a disintegration of the form $\dlambda(t,x) = \dlambda_t(x) \otimes \dt$, where $t\mapsto \lambda_t$ is a bounded (with respect to the total variation norm) measurable map from $[0,T]$ into $\meas^+(\T^d)$. 
			
			\item The energy is bounded by the initial energy, i.e. 
			\begin{equation} \label{eq:euler-mvs-energy} 
				\int_{\T^d} \half \langle \nu_{t,x} , |v|^2 \rangle \dx + \half \lambda_t(\T^d) \leq \int_{\T^d} \half |v_0(x)|^2 \dx\qquad \text{ for a.e. } t\in (0,T).
			\end{equation}
		\end{itemize}
	\end{defn}
	
	\begin{rem} \label{rem:EnIneq->GYMbound}
		Note that \eqref{eq:euler-mvs-energy} makes \eqref{eq:GYM-bound} redundant. For this reason, some authors do not explicitly require \eqref{eq:GYM-bound} when they introduce generalized Young measures in the context of measure-valued solutions to the Euler equations. 
	\end{rem} 
	
	Now we fix $f\in \mathcal{F}_2(\R^d)$, 
	\begin{equation}\label{eq:f-euler}
		f(v)=\half|v|^2,
	\end{equation} 
	which coincides (up to a factor $\half$) with the choice in Rem.~\ref{rem:example-variance}, but also (and more importantly) with the energy for the incompressible Euler equations \eqref{eq:euler-mass}, \eqref{eq:euler-mom}. Consequently, the functional $\V_f$ defined in Defn.~\ref{defn:V} reads 
	$$ 
	\V_f[\nu,\lambda,\nu^\infty] = \int_0^T \int_{\T^d} \half \Big[\langle \nu_{t,x}, |v|^2 \rangle - \big|\langle \nu_{t,x}, v \rangle \big|^2 \Big] \dx\dt + \half \lambda([0,T]\times \T^d).
	$$ 
	
	\begin{rem} \label{rem:incompr-energy} 
		The reader should notice that for the choice $f(v)=\half|v|^2$ we have $f^\infty\equiv \half$. Hence the term 
		$$
		\iint_{[0,T]\times\T^d} \langle \nu^\infty_{t,x}, f^\infty \rangle \dlambda(t,x)
		$$ 
		indeed simplifies to $\half \lambda([0,T]\times \T^d)$. The reason for the term $\half \lambda_t(\T^d)$ in the energy inequality \eqref{eq:euler-mvs-energy} is of the same spirit.
	\end{rem}  
	
	Now let $v_0\in L^2(\T^d)$ arbitrary. We set 
	\begin{align}
		M:= \Big\{ (\nu,\lambda,\nu^\infty)\in \setGYM \,\Big|\,(\nu,\lambda,\nu^\infty)&\text{ is an admissible measure-valued solution } \label{eq:M-euler} \\
		&\text{of \eqref{eq:euler-mass}, \eqref{eq:euler-mom} with initial data }v_0 \Big\}. \notag
	\end{align}
	
	\begin{prop} \label{prop:M-euler-suitable} 
		For the function $f$ chosen in \eqref{eq:f-euler} and the set $M\subset \setGYM$ defined in \eqref{eq:M-euler} the assumptions~\ref{item:A-f-strictlyconvex}-\ref{item:A-M-compact} hold.
	\end{prop}
	
	\begin{proof} 
		\begin{enumerate}
			\item[(A1)] Obviously, $f$ defined in \eqref{eq:f-euler} is strictly convex. 
			
			\item[(A2)] Existence of admissible measure-valued solutions was shown already in \cite{DipMaj87}. We also refer to \name{Brenier}-\name{De~Lellis}-\name{Sz{\'e}kelyhidi}~\cite[Prop.~1]{BreDelSze11} and \cite[Thm.~3.9]{WiedemannPHD} who (in contrast to \name{DiPerna}-\name{Majda}~\cite{DipMaj87}) use the same notation as we do. Consequently, $M\neq \emptyset$. 
			
			\item[(A3)] Let $(\nu^1,\lambda^1,(\nu^\infty)^1),(\nu^2,\lambda^2,(\nu^\infty)^2)\in M$ two admissible measure-valued solutions. It is straightforward to see that their convex combination $(\widehat{\nu},\widehat{\lambda},\widehat{\nu}^\infty)$ as defined in Defn.~\ref{defn:convexity} is a measure-valued solution in the sense of Defn.~\ref{defn:euler-mvs}. Moreover we have 
			\begin{align*}
				\dlambdahat(t,x)&= \tau \dlambda^1(t,x) + (1-\tau) \dlambda^2(t,x) \\
				&= \tau \dlambda^1_t(x) \otimes \dt + (1-\tau) \dlambda^2_t(x) \otimes \dt \\
				&= \Big( \tau \dlambda^1_t(x) + (1-\tau) \dlambda^2_t(x)\Big) \otimes \dt,
			\end{align*}
			i.e.~there exists a desired disintegration of $\widehat{\lambda}$ where $\dlambdahat_t(x) := \tau \dlambda^1_t(x) + (1-\tau) \dlambda^2_t(x)$. The energy balance \eqref{eq:euler-mvs-energy} for the convex combination $(\widehat{\nu},\widehat{\lambda},\widehat{\nu}^\infty)$ is then obvious, which shows that the latter is even admissible. 
			
			\item[(A4)] It is obvious that $M$ is closed with respect to the weak-$\ast$ topology of $\setGYM$. Hence it suffices to show that any sequence $(\nu^k,\lambda^k,(\nu^\infty)^k)_{k\in \N}\subset M$ has a converging subsequence with limit in\footnote{The fact that even this limit lies in $M$ follows from the closedness.} $\setGYM$. Therefore, according to \cite[Cor.~12.3]{Rindler}, it suffices to show that the sequences 
			$$
			\Big(\lambda^k([0,T]\times \T^d)\Big)_{k\in \N}\subset \R\qquad \text{ and }\qquad\left(\int_0^T\int_{\T^d} \langle\nu_{t,x}^k, |\cdot|^2\rangle \dx\dt\right)_{k\in \N}\subset \R
			$$
			are uniformly bounded\footnote{In \cite{Rindler} $p$ is equal to $1$ (see also Rem.~\ref{rem:p}) and consequently \cite[Cor.~12.3]{Rindler} requires boundedness of the sequence $\left(\int_0^T\int_{\T^d} \langle\nu_{t,x}^k, |\cdot|\rangle \dx\dt\right)_{k\in \N}$ instead of $\left(\int_0^T\int_{\T^d} \langle\nu_{t,x}^k, |\cdot|^2\rangle \dx\dt\right)_{k\in \N}$.}. Both bounds hold due to \eqref{eq:euler-mvs-energy}, more precisely
			\begin{align*}
				\lambda^k([0,T]\times\T^d) = \int_0^T (\lambda^k)_t(\T^d) \dt \leq T \|v_0\|_{L^2}^2\quad \text{ and }\quad 
				\int_0^T\int_{\T^d} \langle\nu_{t,x}^k, |v|^2\rangle \dx\dt\leq T \|v_0\|_{L^2}^2
			\end{align*}
			for any $k\in\N$. 
		\end{enumerate}
	\end{proof} 
	
	Combining Thms.~\ref{thm:crit1-existence}, \ref{thm:crit1-uniqueness-meanvalue} with Prop.~\ref{prop:M-euler-suitable} we find the following.
	
	\begin{cor} 
		There exists an admissible measure-valued solution of the incompressible Euler equations \eqref{eq:euler-mass}, \eqref{eq:euler-mom} $(\nu,\lambda,\nu^\infty)\in M$ with initial data $v_0\in L^2(\T^d)$ which is maximal in the sense of Defn.~\ref{defn:crit1} with $f(v)=\half|v|^2$. 
		
		Moreover, any two such maxima $(\nu^1,\lambda^1,(\nu^\infty)^1),(\nu^2,\lambda^2,(\nu^\infty)^2) \in M$ satisfy 
		\begin{align*}
			\langle \nu^1_{t,x}, v \rangle &= \langle \nu^2_{t,x}, v \rangle \qquad \text{ for a.e. } (t,x)\in (0,T)\times\T^d, \text{ and } \\
			\int_0^T \int_{\T^d} \half \langle \nu^1_{t,x}, |v|^2 \rangle \dx \dt + \half \lambda^1([0,T]\times\T^d) &= \int_0^T \int_{\T^d} \half \langle \nu^2_{t,x}, |v|^2 \rangle \dx \dt + \half \lambda^2([0,T]\times\T^d). 
		\end{align*}
	\end{cor}

	\subsection{Compressible Euler equations} \label{subsec:compr-euler} 
	
	As a second example, we consider the isentropic compressible Euler equations
	\begin{align}
		\partial_t \rho + \Div (\rho u) &=0, \label{eq:compr-euler-mass} \\
		\partial_t (\rho u) + \Div( \rho u \otimes u) + \Grad p(\rho) &= 0. \label{eq:compr-euler-mom}
	\end{align}
	Here the unknowns are\footnote{Like in Sect.~\ref{subsec:incomp-euler}, $d\in\N$ denotes the space dimension, $\T^d$ is the $d$-dimensional torus and $T\in (0,\infty)$.} the density $\rho:[0,T)\times \T^d\to \R^+$ and the velocity $u:[0,T)\times \T^d \to \R^d$. Moreover, the pressure $p$ is given by the power law $p(\rho)= \rho^\gamma$ with some $\gamma>1$. The energy density for the isentropic compressible Euler system \eqref{eq:compr-euler-mass}, \eqref{eq:compr-euler-mom} reads
	$$
	\half \rho |u|^2 + \tfrac{1}{\gamma-1} \rho^\gamma.
	$$
	
	\begin{rem}
		Like in \cite{GwiWie15} we exclude vacuum in this paper, i.e.~we only consider solutions with strictly positive density $\rho>0$. A notion of measure-valued solutions to \eqref{eq:compr-euler-mass}, \eqref{eq:compr-euler-mom} which allows for vacuum states (i.e.~$\rho=0$) is available in the literature, see e.g.~\name{Breit}-\name{Feireisl}-\name{Hofmanov{\'a}}~\cite[Defn.~2.1]{BreFeiHof20_1}. 
	\end{rem}
	
	Let us next recall the definition of a measure-valued solution in the context of the compressible Euler equations \eqref{eq:compr-euler-mass}, \eqref{eq:compr-euler-mom}. Such a notion was first introduced by \name{Neustupa}~\cite{Neustupa93}. In order to properly define measure-values solutions of \eqref{eq:compr-euler-mass}, \eqref{eq:compr-euler-mom} in the \name{Alibert}-\name{Bouchitt{\'e}}~\cite{AliBou97} framework, the preliminaries explained in Sect.~\ref{subsec:young-measures} must be slightly modified. We only sketch these refinements here. For more details we refer to \name{Gwiazda}-\name{{\'S}wierczewska-Gwiazda}-\name{Wiedemann}~\cite[Sect.~3]{GwiWie15}. 
	
	Let 
	$$
	\sphere{+}_{\gamma,2} := \left\{ (\beta_1,\beta')\in \R^{1+d} \,\Big|\, |\beta_1|^{2\gamma} + |\beta'|^{4} = 1, \ \beta_1\geq 0 \right\}. 
	$$
	Similar to \cite{GwiWie15}, we use $(\alpha_1,\alpha')\in\R^+\times\R^{d}$ and $(\beta_1,\beta')\in \sphere{+}_{\gamma,2}$ as dummy variables when integrating with respect to $\nu_{t,x}$ and $\nu_{t,x}^\infty$, respectively. One may think of $\alpha_1,\beta_1$ representing $\rho$ and $\alpha',\beta'$ representing $\sqrt{\rho}u$. In this context 
	$$
	\mathcal{A}(\R^+\times\R^{d}) := \left\{ f_0 \in \Cb(\R^+\times\R^{d})\,\Big| \, \lim_{s\to \infty} f_0(s^2\beta_1, s^\gamma \beta') \text{ exists and is continuous in } (\beta_1,\beta')\in \sphere{+}_{\gamma,2}\right\} 
	$$ 
	and 
	\begin{align*}
		\mathcal{F}_{\gamma,2}(\R^+\times\R^{d}) := \Big\{ f \in C(\R^+\times\R^{d})\,\Big| \, \exists &f_0\in \mathcal{A}(\R^+\times\R^{d}) \text{ s.t.~} \\
		&f(\alpha_1,\alpha')=\big(1+(|\alpha_1|^{2\gamma} + |\alpha'|^4)^{1/2}\big) f_0(\alpha_1,\alpha') \Big\} 
	\end{align*} 
	play the role of $\mathcal{A}(\R^m)$ and $\mathcal{F}_2(\R^m)$, respectively. Moreover, the recession function of $f\in \mathcal{F}_{\gamma,2}(\R^+\times\R^{d})$ is given by 
	$$
	f^\infty : \sphere{+}_{\gamma,2} \to \R, \quad f^\infty(\beta_1,\beta'):= \lim_{s\to \infty} f_0(s^2 \beta_1, s^\gamma \beta') = \lim_{s\to \infty} \frac{f(s^2 \beta_1, s^\gamma \beta')}{1+s^{2\gamma}} .
	$$ 
	
	\begin{rem}
		In the context of the compressible Euler equations \eqref{eq:compr-euler-mass}, \eqref{eq:compr-euler-mom} it is also possible to work with variables $\rho$ and $m=\rho u$ (instead of $\rho$ and $\sqrt{\rho} u$), see e.g.~\name{Breit}-\name{Feireisl}-\name{Hofmanov{\'a}}~\cite[Defn.~2.1]{BreFeiHof20_1}.
	\end{rem}
	
	In the compressible setting, a generalized Young measure is an object $(\nu,\lambda,\nu^\infty)$ in 
	\begin{align*}
		\bigg\{ (\nu,\lambda,\nu^\infty) &\in L^\infty_\weak((0,T)\times \T^d; \probmeas(\R^+\times\R^d)) \times \meas^+([0,T]\times \T^d) \times L^\infty_\weak([0,T]\times \T^d,\lambda; \probmeas(\sphere{+}_{\gamma,2})) \, \Big|\, \\
		&\int_0^T \int_{\T^d} \Big[ \langle \nu_{t,x} , |\alpha'|^2 \rangle + \langle \nu_{t,x} , \alpha_1^\gamma \rangle \Big] \dx\dt + \lambda([0,T] \times \T^d) < \infty \bigg\}, 
	\end{align*} 
	which we also denote by $\setGYM$. In particular, the bound \eqref{eq:GYM-bound} has been replaced by 
	\begin{equation} \label{eq:GYM-bound-compr}
		\int_0^T \int_{\T^d} \Big[ \langle \nu_{t,x} , |\alpha'|^2 \rangle + \langle \nu_{t,x} , \alpha_1^\gamma \rangle \Big] \dx\dt + \lambda([0,T] \times \T^d) < \infty . 
	\end{equation} 
	The reader should notice that all statements in Sect.~\ref{sec:general} still hold in this modified setting. 
	
	Now we are ready to write down the definition of a measure-valued solution to the compressible Euler equations.
	\begin{defn}[See {\cite[Sect.~4.1]{GwiWie15}}] \label{defn:compr-euler-mvs}
		Let $\rho_0\in L^\gamma(\T^d)$ and $u_0$ such that $\rho_0 |u_0|^2\in L^1(\T^d)$. A generalized Young measure $(\nu,\lambda,\nu^\infty)\in \setGYM$ is called \emph{measure-valued solution} of the compressible Euler equations \eqref{eq:compr-euler-mass}, \eqref{eq:compr-euler-mom} with initial data $\rho_0,u_0$ if
		\begin{align}
			&\int_0^T \int_{\T^d} \Big[ \langle \nu_{t,x} , \alpha_1\rangle \partial_t \psi + \langle \nu_{t,x} , \sqrt{\alpha_1}\alpha' \rangle \cdot \Grad \psi \Big] \dx\dt + \int_{\T^d} \rho_0 \psi(0,\cdot)\dx = 0, \label{eq:compr-euler-mvs-mass} \\ 
			&\int_0^T \int_{\T^d} \Big[ \langle \nu_{t,x}, \sqrt{\alpha_1}\alpha' \rangle \cdot \partial_t \phi + \langle \nu_{t,x} , \alpha'\otimes \alpha' \rangle : \Grad \phi + \langle \nu_{t,x}, \alpha_1^\gamma \rangle \Div \phi\Big] \dx\dt \notag\\
			&\quad + \int_0^T \int_{\T^d} \Big[ \langle \nu^\infty_{t,x} , \beta'\otimes \beta' \rangle : \Grad \phi + \langle \nu^\infty_{t,x} , \beta_1^\gamma \rangle \Div \phi \Big] \dlambda(t,x) + \int_{\T^d} \rho_0 u_0 \cdot \phi(0,\cdot)\dx = 0 \label{eq:compr-euler-mvs-mom}
		\end{align}
		for all test functions $\psi\in \Cc([0,T)\times \T^d)$, $\phi\in \Cc([0,T)\times \T^d;\R^d)$. 
		
		Such a measure-valued solution $(\nu,\lambda,\nu^\infty)\in \setGYM$ is called \emph{admissible} if the following assertions hold:
		\begin{itemize}
			\item The concentration measure $\lambda$ admits a disintegration of the form $\dlambda(t,x) = \dlambda_t(x) \otimes \dt$, where $t\mapsto \lambda_t$ is a bounded (with respect to the total variation norm) measurable map from $[0,T]$ into $\meas^+(\T^d)$. 
			
			\item The energy is bounded by the initial energy, i.e. 
			\begin{align} 
				&\int_{\T^d} \Big[ \half \langle \nu_{t,x} , |\alpha'|^2 \rangle + \tfrac{1}{\gamma-1} \langle \nu_{t,x} , \alpha_1^\gamma \rangle \Big] \dx + \int_{\T^d} \Big[ \half \langle \nu^\infty_{t,x} , |\beta'|^2 \rangle + \tfrac{1}{\gamma-1} \langle \nu_{t,x}^\infty , \beta_1^\gamma \rangle \Big] \dlambda_t(x) \notag \\ 
				&\leq \int_{\T^d} \Big[ \half \rho_0 |u_0|^2 + \tfrac{1}{\gamma-1} \rho_0^\gamma \Big] \dx\qquad \text{ for a.e. } t\in (0,T). \label{eq:compr-euler-mvs-energy}
			\end{align}
		\end{itemize}
	\end{defn}
	
	\begin{rem}
		Note that in the context of the compressible Euler equations \eqref{eq:compr-euler-mass}, \eqref{eq:compr-euler-mom}, the term 
		$$
		\int_{\T^d} \Big[ \half \langle \nu^\infty_{t,x} , |\beta'|^2 \rangle + \tfrac{1}{\gamma-1} \langle \nu_{t,x}^\infty , \beta_1^\gamma \rangle \Big] \dlambda_t(x)
		$$ 
		cannot be further simplified. This is in contrast to the incompressible setting, see Rem.~\ref{rem:incompr-energy}.
	\end{rem}	
	
	\begin{rem} \label{rem:bound-redundant-compr}
		Like in the incompressible case (see Rem.~\ref{rem:EnIneq->GYMbound}), the energy inequality \eqref{eq:compr-euler-mvs-energy} makes the bound \eqref{eq:GYM-bound-compr} redundant. To see the bound on $\lambda([0,T]\times \T^d)$, we compute 
		\begin{align*} 
			1 &= \int_{\sphere{+}_{\gamma,2}} \Big[|\beta_1|^{2\gamma} + |\beta'|^4\Big] \dd\nu^\infty_{t,x} \leq \int_{\sphere{+}_{\gamma,2}} \Big[|\beta_1|^{\gamma} + |\beta'|^2\Big] \dd\nu^\infty_{t,x} \\
			&\leq \langle \nu^\infty_{t,x} , |\beta'|^2 \rangle + \langle \nu^\infty_{t,x} , \beta_1^\gamma \rangle \\ 
			&\leq \max\{2,\gamma-1\} \Big[ \half\langle \nu^\infty_{t,x} , |\beta'|^2 \rangle + \tfrac{1}{\gamma-1} \langle \nu^\infty_{t,x} , \beta_1^\gamma \rangle \Big],
		\end{align*}
		which holds for $\lambda$-a.e.~$(t,x)\in [0,T]\times \T^d$, or equivalently for a.e.~$t\in [0,T]$ and $\lambda_t$-a.e.~$x\in \T^d$. Integration and the energy bound \eqref{eq:compr-euler-mvs-energy} yield
		\begin{align*} 
			\lambda([0,T]\times \T^d) &= \int_0^T \int_{\T^d} 1 \dd\lambda_t(x) \dt \\
			&\leq \max\{2,\gamma-1\} \int_0^T \int_{\T^d} \Big[ \half\langle \nu^\infty_{t,x} , |\beta'|^2 \rangle + \tfrac{1}{\gamma-1} \langle \nu^\infty_{t,x} , \beta_1^\gamma \rangle \Big] \dd\lambda_t(x) \dt \\
			&\leq \max\{2,\gamma-1\} T \int_{\T^d} \Big[ \half \rho_0 |u_0|^2 + \tfrac{1}{\gamma-1} \rho_0^\gamma \Big] \dx.
		\end{align*}
	\end{rem} 
	
	\begin{rem}
		One may treat $(t,x)\mapsto \langle \nu_{t,x} , \alpha_1\rangle$ and $(t,x)\mapsto\langle \nu_{t,x}, \sqrt{\alpha_1}\alpha' \rangle$ as being continuous in time with respect to the weak-$\ast$ topology in $L^\infty$. In fact, the above maps can be redefined on a set of times of measure zero so that they have this continuity property, see e.g.~\cite[Thm.~4.1.1]{Dafermos1stEd} or \cite[Lemma~1 and Appendix~A]{DelSze10}. As a consequence, \eqref{eq:compr-euler-mvs-mass}, \eqref{eq:compr-euler-mvs-mom} can be equivalently formulated in a slightly different way, namely with integrals in time only over an interval $[0,\tau]$ (rather than $[0,T]$) for any $0\leq \tau \leq T$ and with an additional end condition at $t=\tau$. This is the way how \eqref{eq:compr-euler-mvs-mass}, \eqref{eq:compr-euler-mvs-mom} are written in \cite{GwiWie15}, see equation (4.2) therein.
	\end{rem}
	
	In the context of the compressible Euler system \eqref{eq:compr-euler-mass}, \eqref{eq:compr-euler-mom}, we choose $f\in \mathcal{F}_{\gamma,2}(\R^+\times\R^{d})$, 
	\begin{equation} \label{eq:f-compr-euler}
		f(\alpha_1,\alpha')=\half|\alpha'|^2 + \tfrac{1}{\gamma-1} \alpha_1^\gamma, 
	\end{equation}
	whose recession function reads $f^\infty(\beta_1,\beta')=\half|\beta'|^2 + \tfrac{1}{\gamma-1} \beta_1^\gamma$. 
	Like in the incompressible case, $f$ represents the energy. Thus, the functional $\V_f$ defined in Defn.~\ref{defn:V} is given by
	\begin{align*}
		\V_f[\nu,\lambda,\nu^\infty] &= \int_0^T \int_{\T^d} \Big[ \half \langle \nu_{t,x}, |\alpha'|^2 \rangle + \tfrac{1}{\gamma-1} \langle \nu_{t,x}, \alpha_1^\gamma \rangle - \half \big|\langle \nu_{t,x}, \alpha' \rangle \big|^2 - \tfrac{1}{\gamma-1} \langle \nu_{t,x}, \alpha_1 \rangle^\gamma \Big] \dx\dt \\
		&\quad + \iint_{[0,T]\times\T^d} \Big[ \half\langle \nu^\infty_{t,x} , |\beta'|^2 \rangle + \tfrac{1}{\gamma-1} \langle \nu_{t,x}^\infty , \beta_1^\gamma \rangle \Big] \dlambda(t,x). 
	\end{align*}
	
	Again for given initial data  $\rho_0\in L^\gamma(\T^d)$ and $u_0$ such that $\rho_0 |u_0|^2\in L^1(\T^d)$, we set 
	\begin{align}
		M:= \Big\{ (\nu,\lambda,\nu^\infty)\in \setGYM \,\Big|\,(\nu,\lambda,\nu^\infty)&\text{ is an admissible measure-valued solution } \label{eq:M-compr-euler} \\
		&\text{of \eqref{eq:compr-euler-mass}, \eqref{eq:compr-euler-mom} with initial data }\rho_0,u_0 \Big\}. \notag
	\end{align}
	
	\begin{prop} \label{prop:M-compr-euler-suitable} 
		For the function $f$ chosen in \eqref{eq:f-compr-euler} and the set $M\subset \setGYM$ defined in \eqref{eq:M-compr-euler} the assumptions~\ref{item:A-f-strictlyconvex}-\ref{item:A-M-compact} hold.
	\end{prop} 
	
	\begin{proof} 
		\begin{enumerate} 
			\item[(A1)] It is obvious that $f$ given by \eqref{eq:f-compr-euler} is strictly convex. 
			
			\item[(A2)] Existence of measure-valued solutions for the compressible Euler system \eqref{eq:compr-euler-mass}, \eqref{eq:compr-euler-mom} goes back to \name{Neustupa}~\cite{Neustupa93}. We also refer to \name{Gwiazda}-\name{{\'S}wierczewska-Gwiazda}-\name{Wiedemann}~\cite[Rem.~4.1]{GwiWie15}. Thus, $M\neq \emptyset$. 
			
			\item[(A3)] The convexity of $M$ can be shown exactly as in the incompressible case, see the proof of Prop.~\ref{prop:M-euler-suitable}.
			
			\item[(A4)] To prove that $M$ is sequentially compact, one proceeds like in the incompressible case, see the proof of Prop.~\ref{prop:M-euler-suitable}. In particular, one has to prove that the sequences
			$$
			\Big(\lambda^k([0,T]\times \T^d)\Big)_{k\in \N}\subset \R\qquad \text{ and }\qquad\left(\int_0^T \int_{\T^d} \Big[ \langle \nu_{t,x}^k , |\alpha'|^2 \rangle + \langle \nu_{t,x}^k , \alpha_1^\gamma \rangle \Big] \dx\dt\right)_{k\in \N}\subset \R
			$$
			are uniformly bounded. Similarly to the bound \eqref{eq:GYM-bound-compr}, this immediately follows from \eqref{eq:compr-euler-mvs-energy}, see Rem.~\ref{rem:bound-redundant-compr}. 
		\end{enumerate} 
	\end{proof}
	
	As a consequence of Thms.~\ref{thm:crit1-existence}, \ref{thm:crit1-uniqueness-meanvalue} as well as Prop.~\ref{prop:M-compr-euler-suitable} we obtain the following.
	
	\begin{cor}
		There exists an admissible measure-valued solution of the compressible Euler equations \eqref{eq:compr-euler-mass}, \eqref{eq:compr-euler-mom} $(\nu,\lambda,\nu^\infty)\in M$ with initial data  $\rho_0\in L^\gamma(\T^d)$ and $u_0$ such that $\rho_0 |u_0|^2\in L^1(\T^d)$ which is maximal in the sense of Defn.~\ref{defn:crit1} with $f(\alpha_1,\alpha')=\half|\alpha'|^2 + \tfrac{1}{\gamma-1} \alpha_1^\gamma$.
		
		Moreover, any two such maxima $(\nu^1,\lambda^1,(\nu^\infty)^1),(\nu^2,\lambda^2,(\nu^\infty)^2) \in M$ satisfy 
		\begin{align*}
			&\langle \nu^1_{t,x}, \alpha_1 \rangle = \langle \nu^2_{t,x}, \alpha_1 \rangle \quad \text{ and }\quad \langle \nu^1_{t,x}, \alpha' \rangle = \langle \nu^2_{t,x}, \alpha' \rangle \qquad \text{ for a.e. } (t,x)\in (0,T)\times\T^d, \text{ and } \\ 
			&\int_0^T \int_{\T^d} \Big[ \half \langle \nu_{t,x}^1, |\alpha'|^2 \rangle + \tfrac{1}{\gamma-1} \langle \nu_{t,x}^1, \alpha_1^\gamma \rangle \Big] \dx\dt \\
			&\quad + \iint_{[0,T]\times\T^d} \Big[ \half\langle (\nu^\infty)_{t,x}^1 , |\beta'|^2 \rangle + \tfrac{1}{\gamma-1} \langle (\nu^\infty)_{t,x}^1 , \beta_1^\gamma \rangle \Big] \dlambda^1(t,x) \\
			&= \int_0^T \int_{\T^d} \Big[ \half \langle \nu_{t,x}^2, |\alpha'|^2 \rangle + \tfrac{1}{\gamma-1} \langle \nu_{t,x}^2, \alpha_1^\gamma \rangle \Big] \dx\dt \\
			&\quad + \iint_{[0,T]\times\T^d} \Big[ \half\langle (\nu^\infty)_{t,x}^2 , |\beta'|^2 \rangle + \tfrac{1}{\gamma-1} \langle (\nu^\infty)_{t,x}^2 , \beta_1^\gamma \rangle \Big] \dlambda^2(t,x). 
		\end{align*} 
	\end{cor}
	
	\begin{rem} \label{rem:DMV-solution} 
		In the context of the isentropic Euler equations \eqref{eq:compr-euler-mass}, \eqref{eq:compr-euler-mom}, there is another definition of measure-valued solutions avaliable in the literature: the so-called \emph{dissipative measure-valued (DMV) solutions}, see \cite{BreFeiHof20_1}. These DMV solutions do not exactly fit into the \name{Alibert}-\name{Bouchitt{\'e}}~\cite{AliBou97} framework outlined in Sect.~\ref{sec:general}. One of the main differences between admissible measure-valued solutions in the sense of Defn.~\ref{defn:compr-euler-mvs} and DMV solutions lies in the defect terms. It is however not difficult to modify the functional $\V_f$ in order to apply our theory to DMV solutions. In particular, one has to replace the term 
		$$
		\iint_{[0,T]\times\T^d} \Big[ \half\langle \nu^\infty_{t,x} , |\beta'|^2 \rangle + \tfrac{1}{\gamma-1} \langle \nu_{t,x}^\infty , \beta_1^\gamma \rangle \Big] \dlambda(t,x)
		$$ 
		by its analogue in the DMV setting.
	\end{rem}

	\section*{Acknowledgements}
	
	C.K.~acknowledges funding by the Deutsche Forschungsgemeinschaft (DFG, German Research Foundation) within SPP 2410 \emph{Hyperbolic Balance Laws in Fluid
		Mechanics: Complexity, Scales, Randomness (CoScaRa)}, project number 525941602.
	
	S.M.~acknowledges financial support from the Alexander von Humboldt foundation and also from the Deutsche Forschungsgemeinschaft (DFG, German Research Foundation) within SPP 2410, project number 525935467.
	
	E.W.~acknowledges funding by the Deutsche Forschungsgemeinschaft (DFG, German Research Foundation) within SPP 2410, project number 525716336.  


\end{document}